\documentclass[12pt]{article}

\usepackage[utf8]{inputenc}
\usepackage[T1]{fontenc}
\usepackage{amsfonts}
\usepackage{amsmath}
\usepackage{amsbsy}
\usepackage{amstext}
\usepackage{amsxtra}
\usepackage{amsthm}
\let\origlrcorner\lrcorner 
\usepackage{mathabx}
\let\lrcorner\origlrcorner
\usepackage{amssymb}
\usepackage[heavycircles]{stmaryrd}
\usepackage{tikz-cd}
\usepackage{graphicx}
\usepackage{enumitem}
\usepackage[margin=2.5cm]{geometry}
\usepackage[colorlinks=true,pdftex,bookmarksopen,colorlinks]{hyperref}

\numberwithin{equation}{section}

\swapnumbers
\theoremstyle{definition}
\newtheorem{defn}{Definition}[section]
\newtheorem{obs}[defn]{Remark}
\newtheorem{exple}[defn]{Example}
\newtheorem{exples}[defn]{Examples}
\theoremstyle{plain}
\newtheorem{teo}[defn]{Theorem}

\newtheorem{prop}[defn]{Proposition}

\def\Z{\mathbb{Z}}
\def\R{\mathbb{R}}

\def\ins{\lrcorner}

\renewcommand{\d}{\mathbf{d}}

\newcommand{\abs}[1]{\left|{#1}\right|}

\def\Der{\operatorname{Der}}
\def\Id{\operatorname{Id}}

\def\lsem{\llbracket}
\def\rsem{\rrbracket}
\newcommand\superfold[1]{\left({#1}\middle|\mathcal{R}{#1}\right)}
\newcommand\supermap[2]{\left({#1}\middle|{#2}\right)}
\newcommand{\parity}[1]{\lfloor{#1}\rceil}

\newcommand{\inv}[1]{{#1}^{\operatorname{-}\!1}}
\newcommand{\set}[2][\empty]{\left\{{#2}\right\}_{#1}}
\newcommand{\optupla}[4][1]{{#2}_{#1}{#3}\cdots{#3}{#2}_{#4}}
\newcommand{\omitupla}[5][1]{{#2}_{#1}{#3}\cdots{#3}
\widehat{{#2}_{#4}}{#3}\cdots{#3}{{#2}_{#5}}}

\newcommand\sect[2][\empty]{\Gamma_{#1}\left({#2}\right)}
\newcommand\C{\mathcal{C}^{\infty}}

\begin{document}
\title{On Lie superalgebroids of Poisson type}
\author{Dennise García-Beltrán \and Óscar 
Guajardo}
\date{}
\maketitle

\begin{abstract}
The purpose of this work is to study Lie superalgebroid structures on the space 
of superdifferential $1$-forms over a special kind 
of supermanifolds ---which we call \textit{cotangent supermanifolds}--- whose 
superfunctions are the differential forms on its underlying manifold. These 
superalgbroids are constructed from graded Poisson structures defined on the 
latter superalgebra.
\end{abstract}



\section*{Introduction}
Poisson manifolds and Lie algebroids are two well-studied structures in differential 
geometry, both for their applications in other areas of mathematics, as well as for their 
usefulness in describing physical phenomena; it is also well-known the interplay between 
these two structures (see for instance~\cite{GVV},  \cite{ks-poisson}, and 
references therein). 

On the other hand, given a Lie algebroid there is associated to it another algebraic 
structure, a \emph{Gerstenhaber algebra}, whose prototype is the $\Z$-graded algebra of 
sections of multivectors on a manifold $M$, endowed with both the wedge product and the 
Schouten-Nijenhuis bracket $[\cdot,\cdot]_{\operatorname{SN}}$; this bracket is the only 
extension to this space that restricts to the usual Lie bracket of vector fields, endows the 
former with the structure of a Lie algebra (actually, of a $\Z$-graded Lie algebra), and 
which is also compatible with the wedge product, in the sense that maps of the form 
$[P,\cdot]_{\operatorname{SN}}$ are derivations of the wedge product, for any section $P$ of 
$\Lambda TM$. These relations have also been well-studied---cf. e.g.~\cite{ks-exact}, 
\cite{xu}. Gerstenhaber algebras have also been used in theoretical physics, specifically in 
the study of BRST quantisation (see for example~\cite{ks-odd}). In a related direction, 
graded Poisson structures are also of interest, not only as a generalisation of (ungraded) 
Poisson structures to graded and supermanifolds, but also as algebraic structures relevant 
for the aforementioned studies---cf.~\cite{BM}, \cite{ci-even}, for instance.

In this paper we consider Poisson structures on the graded-commutative algebra 
of differential forms $\mathcal{A}=\Omega(M)$ of a smooth manifold $M$. We 
endow the pair $(M,\mathcal{A})$ with the structure of a smooth 
supermanifold and consider superalgebroids defined on the space 
$\Omega^{1}(\mathcal 
A)$ of graded (or super-) differential $1$-forms on this supermanifold. More 
specifically, we consider superalgebroid structures defined using the given 
Poisson bracket on $\mathcal A$.

The structure of the paper is as follows: in Section \ref{sec:pre} we define 
the concepts relevant to this work; to wit: supermanifolds, superbundles and 
superalgebroids. In Section \ref{sec:pois} we introduce graded Poisson 
structures on a superalgebra $\mathcal{A}$ and study some of their properties; 
in particular we construct Poisson-type superalgebroids on the space 
$\Omega^{1}(\mathcal A)$, which is defined as a certain subspace of 
$\mathcal{A}\otimes\mathcal{A}$; we also characterise these superalgebroids. In 
Section \ref{sec:gerst} we prove the equivalence between a Lie superalgebroid structure on 
$\Omega^{1}(\mathcal{A})$ and a Gerstenhaber superalgebra (suitably defined in this section) 
on graded forms $\Omega(\mathcal A)$; this result is well-known in the classical case (see, 
for instance, \cite{ksm}), but it seems to be lacking in the supergeometric setting, which 
is the reason why we write down a proof. This last section also establishes the equivalence 
between \emph{differential} Gerstenhaber superalgebras and certain Poisson superstructures.

\section{Preliminaries}\label{sec:pre}
	
In this work, the category of (real) supermanifolds will be understood in the sense of 
\cite{guajardo}; to wit, a supermanifold of dimension $(m|n)$ is a pair 
$\superfold{M}$ 
where $M$ is a smooth manifold of dimension $m$ and $\mathcal{R}M$ is a superalgebra bundle 
over $M$ such that each fibre $\mathcal{R}_{p}M$ (for any $p$ in $M$) is isomorphic to the 
exterior algebra $\Lambda{S}^{*}$ for some (real) vector space $S$ of dimension $n$; the 
\textbf{smooth superfunctions} are the sections of the bundle $\mathcal{R}M$, i.e., they are 
the elements of the supercommutative unital algebra  $\sect{\mathcal{R}M}$. 

By a \textbf{supervector bundle} over $\superfold{M}$ we mean simply a vector 
bundle $E$ 
over $M$ such that
\begin{enumerate}
	\item $E=E_{0}\oplus{E_1}$, and	
	\item each fibre $E_p$ carries the structure of a (left) $\mathcal{R}_{p}M$-supermodule.
\end{enumerate}
We refer the reader to \cite{adolfo-bundles} for a comprehensive study of 
supervector bundles. Do note that the objects studied there are quite more 
general than what we have called a supervector bundle; for our purposes, we 
need only consider the supermodule structure on the space of sections of 
$E$---and not a supermanifold structure on its total space---, so our 
definition is general enough, but the reader should bear in mind the scope of 
the usual definition of a supervector bundle. The space of sections \emph{over 
$M$} of $E$ is then a 
$\sect{\mathcal{R}M}$-supermodule. We will consider a special kind of these:

\begin{defn}
	Let $\mathcal{A}$ be a supercommutative algebra. A \textbf{standard (left) supermodule} 
	is a bilateral supermodule $N$ such that $an=(-1)^{\parity{a}\parity{n}}na$ for any 
	homogeneous $a\in\mathcal{A}$ and $n\in{N}$.
\end{defn}
There is of course an analogous definition of a right standard supermodule. 
Throughout this work, the symbol $\parity{\cdot}$ denotes the parity of a 
homogeneous element of a superalgebra. It will be tacitly assumed that all 
supervector bundles arise from standard supermodules. In particular we adopt 
the convention that the module of supervector fields is a left standard module 
(over the superalgebra $\mathcal{A}=\sect{\mathcal{R}M}$), whereas the space of 
superforms $\Omega\superfold{M}$ is a bigraded superalgebra (with bigrade over 
$\Z_{2}\times\Z$) and a right standard module; we will refer to the 
$\Z_2$-grading as the \emph{parity} of a bihomogeneous superform, and to its 
$\Z$-grading as the \emph{degree}.

We are interested mostly on the supermanifolds whose smooth superfunctions are 
the differential forms on ordinary manifolds, that is where 
$\mathcal{R}M=\Lambda{T}^{*}M$.

\begin{defn}
	The supermanifold $\supermap{M}{\Lambda{T^*}M}$ is the \textbf{cotangent 
	supermanifold} on $M$. We will denote it by $\supermap{M}{TM}$ (cf. 
	\cite[Section 1.4]{guajardo} for this notation and its meaning).
\end{defn}

In what follows, in order to save space, we shall denote a supermanifold by 
$\mathcal{M}=\superfold{M}$, its tangent superbundle by $T\mathcal{M}$, and a supervector 
bundle simply by $E$. 

\begin{defn}\label{def:salgebroide}
	Let $\mathcal M$ be a supermanifold and $E$ a supervector bundle over $\mathcal M$. A 
	\textbf{superalgebroid structure}
	on $E$ is an even bundle map $\rho : \Gamma(E) \rightarrow \Gamma(T\mathcal{M})\simeq 
	\Der(\Gamma(\mathcal{R}M))$ (the \emph{anchor}) and a
	bilinear operation $\lsem\cdot, \cdot\rsem : \Gamma(E) \times \Gamma(E) \rightarrow 
	\Gamma(E)$ (the \emph{bracket}) such that
	\begin{enumerate}
		\item $\lsem\cdot,\cdot\rsem$ gives $\Gamma(E)$ the structure of a Lie superalgebra, 
		and
		\item $\lsem X, f Y \rsem = \rho(X) (f)  Y + (-1)^{\parity{X}\parity{f}} f \lsem X, 
		Y 
		\rsem$ for any $X, Y \in \Gamma(E)$,	$f \in \sect{\mathcal{R}M}$ 
		(graded 
		Leibniz	
		identity).
	\end{enumerate}
\end{defn}

\begin{exples}
	\begin{enumerate}
		\item The supertangent bundle $T\mathcal{M}$ is a Lie superalgebroid 
		with $\rho=\Id $ and whose brackets is the usual supercommutator of supervector 
		fields.
		\item If $L\colon{T}\mathcal{M}\to{T^*}\mathcal{M}$ is a bundle 
		isomorphism, then $T^{*}\mathcal{M}$ becomes a Lie superalgebroid by 
		defining $\rho=\inv{L}$ and $\lsem\alpha,\beta\rsem=
		L([\inv{L}(\alpha),\inv{L}(\beta)])$, where $[\cdot,\cdot]$ is the same bracket of the previous example. 
	\end{enumerate}
\end{exples}

\section{Poisson-type superalgebroids}\label{sec:pois}
We now turn our attention to special sorts of $\R$-algebras $\mathcal{A}$.

\begin{defn}
	A \textbf{graded Poisson algebra} of degree $k$ $(\mathcal{A}, \{\cdot ,\cdot \})$ 
	is a $\Z$-graded associative and graded-commutative algebra 
	$(\mathcal{A}^{\bullet},\cdot)$ together with 
	an $\mathbb{R}$-bilinear operation, the \textbf{Poisson bracket},
	$\{\cdot ,\cdot \} : \mathcal{A} \times \mathcal{A}\rightarrow \mathcal{A}$ such that
	\begin{enumerate}
		\item $\{\mathcal{A}^{a},\mathcal{A}^{b}\}\subseteq\mathcal{A}^{a+b+k}$ (the bracket 
		has degree $k$);
		\item $\{f, g\} = -(-1)^{(\abs{f}+k)(\abs{g}+k)}\{g, f \}$ (graded skew-symmetry);
		\item $\{f, \{g, h\}\}= \{\{f, g \},h\} +(-1)^{(\abs{f}+k)(\abs{g}+k)} \{g, 
		\{f, 
		h\}\} $ (graded Jacobi identity), and
		\item $\{f, gh\} = \{f, g\}h +(-1)^{(\abs{f}+k)(\abs{g})} g\{f, h\}$ 
		(graded Leibniz identity),
	\end{enumerate}
	for all $f, g, h \in\mathcal{A}$. Here $\abs{\cdot}$ denotes the degree of an element of 
	$\mathcal A$.
\end{defn}

Note that a 
graded-commutative algebra furnishes a supercommutative algebra by defining the parity of a 
$\Z$-homogeneous element as the parity of its degree; in this case, parity and 
degree are 
indistinguishable. 

For any graded algebra $\mathcal{A}$, the 	space $\Omega^{1}(\mathcal{A})$ is 
defined as 
the kernel of the multiplication map 
$\mathbf{m}\colon\mathcal{A}\otimes\mathcal{A}\to\mathcal{A}$; let us define 
\begin{equation}\label{eq:superd}
\begin{split}
\d\colon\mathcal{A}&\to\Omega^{1}(\mathcal{A})\\
a		&\mapsto a\otimes{1}-1\otimes{a}
\end{split}
\end{equation}
and observe that $\Omega^{1}(\mathcal{A})$ is generated by the image of this 
map; also, note that $\d$ is a derivation with values in the latter module, 
which is the dual module to the one of derivations $\Der\mathcal{A}$  with the 
pairing
\begin{equation}\label{eq:pairing}
\langle D,\d{a}\rangle=
D(a)
\end{equation}
(all this 
is in accordance with \cite[p.3]{GVV}).
\begin{obs}
	The module $\mathcal{A}\otimes\mathcal{A}$ is technically a bimodule with 
	left and right multiplications
	\[
	a\otimes(\alpha\otimes\beta)=
	a\alpha\otimes\beta
	\quad\text{and}\quad
	(\alpha\otimes\beta)\otimes{a}=
	\alpha\otimes\beta{a}
	\]
	respectively. Since the definition of $\Omega^{1}(\mathcal A)$ renders it 
	as a submodule of the one above, we can consider it to be a bilateral 
	module on its own, \emph{but a right module} with respect to the pairing 
	\eqref{eq:pairing}. However, since our definition of a superalgebroid 
	implicitly assumes such an object to be a left module, we shall treat 
	$\Omega^{1}(\mathcal A)$ as a left module whenever necessary, considering 
	it as a submodule of $\mathcal{A}\otimes\mathcal{A}$; this is done in order 
	to 	avoid clumsy computations below.
\end{obs}

Let us note, 
however, that since we shall 
be working with $\mathcal{A}=\sect{\mathcal{R}M}$, the space 
$\Omega^{1}(\mathcal{A})$ is in 
fact the space of $1$-superforms on $\mathcal M$ and the map $\d$ is 
simply the 
exterior superderivative, which is even and has degree $1$; its definition 
agrees with the above.

\begin{exples}\label{ex:ejemplos}\hspace{1pt}
\begin{enumerate}
	\item Any Poisson algebra in the classical sense is a graded Poisson algebra whose 
	bracket has degree $0$.
 	\item\label{ex:euclidean} If $(V,g)$ is a pseudoeuclidean vector space then defining on 
	the 
	exterior algebra 
	$\Lambda V$ the bracket $\{\cdot,\cdot\}$ by the relations
	\begin{enumerate}
		\item $\{1,\cdot\}=0$;
		\item $\{x,y\}=g(x,y)$ for $x$ and $y$ in $V$, and
		\item $\{x,y\wedge z\}=\{x,y\}z-\{x,z\}y$ for all $x,y$ and $z$ in $V$,
	\end{enumerate}
	we obtain a graded Poisson algebra of degree $-2$.
	\item If $(M,(\cdot,\cdot))$ is a Poisson manifold, then by defining on 
	$\mathcal{A}=\Omega(M)$ 
	the bracket $\{\cdot,\cdot\}$ by the relations
	\begin{enumerate}
		\item $\{a,b\}=0$ for $a,b\in\Omega^{0}(M)$;
		\item $\{{d}a,b\}=\{a,{d}b\}=(a,b)$ for $a,b\in\Omega^{0}(M)$;
		\item $\{{d}a,{d}b\}={d}(a,b)$, and
		\item $\{\alpha,\beta\wedge\gamma\}=\{\alpha,\beta\}\wedge\gamma + 
		(-1)^{(\abs{\alpha}-1)\abs{\beta}}\beta\wedge\{\alpha,\gamma\}$ for all 
		$\alpha,\beta,\gamma$ in $\Omega^{1}(M)$,
	\end{enumerate}
	we obtain a graded Poisson algebra on $\mathcal{A}$ of degree $-1$. This 
	bracket was completely characterised in \cite{BM}.
\end{enumerate}	
\end{exples}

If $(\mathcal{A},\{\cdot,\cdot\})$ is graded Poisson algebra of degree $k$, we can define an 
$\mathcal{A}$-module morphism $\Omega^1(\mathcal{A})\rightarrow \Der (\mathcal{A})$ by 
$\rho(\d{f})=\{f,\cdot\}$ for $f\in\mathcal{A}$ (where $\mathbf d$ is 
defined by \eqref{eq:superd}),  and extending by 
$\mathcal{A}$-linearity. 
This map furnishes a derivation of degree $k$ on $\mathcal A$.

We can also define a bracket 
$\lsem\cdot,\cdot\rsem:\Omega^1(\mathcal{A})\times\Omega^1(\mathcal{A})\rightarrow 
\Omega^1(\mathcal{A})$ by 
\begin{eqnarray}
\lsem\alpha,\beta\rsem=\rho(\alpha) \ins \d\beta 
-(-1)^{(\parity{\alpha}+k)(\parity{\beta}+k)} \rho(\beta) \ins 
\d\alpha + \d(\beta(\rho(\alpha)))
\end{eqnarray}
where $\ins$ denotes interior product and which satisfies
\begin{enumerate}
	\item 
	$\lsem\alpha,\beta\rsem=-(-1)^{(\parity{\alpha}+k)(\parity{\beta}+k)}
	\lsem\beta,\alpha\rsem$
	\item 
	$\lsem\alpha,\lsem\beta,\gamma\rsem\rsem=
	\lsem\lsem\alpha,\beta\rsem,\gamma\rsem+
	(-1)^{(\parity{\alpha}+k)(\parity{\beta}+k)}\lsem\beta,\lsem\alpha,\gamma\rsem\rsem$
	\item 
	$\lsem\alpha,f\beta\rsem=
	\rho(\alpha)(f)\beta+
	(-1)^{(\parity{\alpha}+k)\parity{f}}f\lsem\alpha,\beta\rsem$
\end{enumerate}
for all $\alpha,\beta\in \Omega^1(\mathcal{A})$ and $f\in\mathcal{A}$, which 
can 
be verified with direct computations.

\begin{defn}\label{def:sPoisson}
When $(\mathcal{A}, \{\cdot ,\cdot \})$	is a graded Poisson algebra of even degree, then 
$(\mathcal{A}, \{\cdot ,\cdot \})$ is called an \textbf{Poisson 
superalgebra}.
\end{defn}

Note that if $k$ is even then $\parity{\parity{a}+k}=\parity{a}$ for all 
$\alpha$ in $\Omega^{1}(\mathcal{A})$, so it suffices to consider the case
$k=0$.

\begin{teo}\label{teo:algebra} 
If $(\mathcal{A}, \{ \cdot,\cdot\})$ is a Poisson superalgebra, then $(\Omega ^1 
(\mathcal{A}), \lsem\cdot ,\cdot \rsem, \rho)$ is a Lie superalgebroid with anchor map 
$\rho$ 
defined by $\rho(\d{f})=\{f,\cdot\}$ for $f\in \mathcal{A}$ and 
bracket defined 
by 
\begin{equation}\label{def:corchete}
	\lsem\alpha,\beta\rsem=\rho(\alpha) \ins \d\beta 
	-(-1)^{\parity{\alpha}\parity{\beta}} \rho(\beta) \ins 
	\d\alpha + \d(\beta(\rho(\alpha)))
\end{equation}
for every $\alpha, \beta\in \Omega ^1(\mathcal{A})$.
\end{teo}
\begin{proof}
All the required properties can be checked directly from the definition 
of the bracket 
\eqref{def:corchete}.
To verify the Jacobi identity, it is easier to consider a system $\set[i\in 
I]{\d{f_i}}$ of 
generators of $\Omega^1(\mathcal{A})$, and to express the elements $\alpha \in 
\Omega^1(\mathcal{A})$ as $\alpha=\sum_{i}{g_i}(\d{f_i})$ for $g_i \in 
\mathcal{A}$.
\end{proof}


\begin{exple}
If $(M, g)$ is a Riemannian manifold, we can define 
\begin{align}
\{\phi,\cdot\}&=0\\
\{\mu,\nu\}&=g(\mu^{\sharp},\nu^{\sharp})\\
\{\mu,\nu\wedge\xi\}&=g(\mu^{\sharp},\nu^{\sharp})\xi-g(\mu^{\sharp},\xi^{\sharp})\nu
\end{align} 
for $\phi\in \C(M)$, $\mu,\nu, \xi\in \Omega^1(M) $; here, $\alpha^\sharp$ is the vector field associated to a $1$-form $\alpha$ via the \emph{musical isomorphism} induced by $g$.
It is easy to see that $ \{\cdot, \cdot\} $ defines a structure of 
graded Poisson algebra of degree $ -2 $ on $ \mathcal{A}=\Omega (M)$ (in fact,
\ref{ex:euclidean} of \ref{ex:ejemplos} is a particular case of this example, when $ M = \{* 
\} $ and $ (M, g) $ 
is a pseudoriemannian manifold). Then, by Theorem \ref{teo:algebra}, $(\Omega ^1 
(\mathcal{A}), \lsem\cdot ,\cdot \rsem, \rho)$ is a Lie superalgebroid with anchor map 
$\rho$  defined by $\rho(\d{f})=\{f,\cdot\}$ for $f\in \mathcal{A}$ and 
bracket defined by \eqref{def:corchete} for $\alpha,\beta\in \Omega^1(\mathcal{A})$.

More explicitly, if $\alpha=f_{1}\d{h_1}$ and $\beta=f_{2}\d{h_2}$ with 
$f_1,f_2,h_1,h_2\in\mathcal{A}$, then
\begin{equation}\label{eq:ejRiemann}
\begin{split}
\lsem\alpha,\beta\rsem
&=f_1\{ 
h_1,f_2\}\d{h_2}-
(-1)^{\parity{\alpha}\parity{\beta}}f_2\{ 
h_2,f_1\}\d{h_1}\\
&-(-1)^{\parity{\alpha}\parity{\beta}+\parity{h_2}\parity{f_1}}f_2
 f_1\d\{h_2,h_1\}
\end{split}
\end{equation}
\end{exple}

Theorem \ref{teo:algebra} tells us that given a Poisson superalgebra we have a Lie 
superalgebroid naturally associated to it. We are interested now in the 
reciprocal question: When 
does a Lie superalgebroid $(\Omega^1(\mathcal{A}), \lsem \cdot,\cdot \rsem, \rho)$
determine a Poisson superalgebra on $\mathcal{A}$?

In order to answer this question, let  us consider a Lie superalgebroid 
$(\Omega^1(\mathcal{A}), \lsem \cdot,\cdot \rsem, \rho)$; for any 
$f, h \in \mathcal{A}$ we can define
\begin{equation}\label{eq:corcheteP}
\{f, h\} = (\d{h})(\rho(\d{f} )) = (\rho(\d{f} ))(h),
\end{equation}
which is clearly $\mathbb{R}$-bilinear and 
satisfies the Leibniz identity
\begin{equation*}
\begin{split}
\{f, hh'\}& = (\rho(\d{f} ))(hh') = \rho(\d{f} )(h)h' 
+(-1)^{\parity{f}\parity{h}}h\rho(\d{f} )(h') \\
& = \{f, h\}h' + (-1)^{\parity{f}\parity{h}} h\{f, h'\},
\end{split}
\end{equation*}

for all $f, h, h' \in \mathcal{A}$ homogeneous.

Thus, in order to get a Poisson superalgebra on $\mathcal{A}$ we only need to take care of 
the skew-supersymmetry and the graded Jacobi identity.

\begin{defn} 
The anchor map $\rho$ of a Lie superalgebroid $(\Omega^1(\mathcal{A}), 
\lsem\cdot,\cdot\rsem,\rho)$ is said to be \textbf{skew-supersymmetric} if
\begin{equation}\label{eq:superskew}
\alpha(\rho(\beta)) = -(-1)^{\parity{\alpha}\parity{\beta}}\beta(\rho(\alpha))
\end{equation}
for all $\alpha, \beta\in \Omega^1(\mathcal{A})$.
\end{defn}

Observe that if the anchor map is skew-supersymmetric, then the new operation 
defined by \eqref{eq:corcheteP} is also skew-supersymmetric, so henceforth we 
will assume $\rho$ to satisfy \eqref{eq:superskew}. Let us now turn our 
attention to 
the graded Jacobi identity.

\begin{teo}\label{teo:sn}
Let $(\Omega^1(\mathcal{A}), \lsem \cdot,\cdot \rsem, \rho)$ be a Lie superalgebroid, and 
define $Q \in \Lambda ^ 2 (\Der(\mathcal{A}))$ by $Q(\d{f}, \d{g}) 
=\d{g}(\rho(\d{f} ))$ for $f, g \in 
\mathcal{A}$. The following conditions are equivalent:
\begin{enumerate}
\item\label{1} the operation defined by $\{f, g\} = \rho(\d{f} )(g)$ 
satisfies the Jacobi identity,
\item\label{2} $[Q, Q]_{\operatorname{SN}} = 0$,
\item\label{3} $[\rho(\d{f} ), \rho(\d{g})] = \rho(\d\{f, g\})$, where 
$\{\cdot,\cdot\}$ is defined in \ref{1}.
\end{enumerate}
Here, $[\cdot, \cdot]_{SN}$ denotes the generalisation of the 
Schouten–Nijenhuis bracket to elements of the bigraded algebra 
$\Lambda\Der\mathcal{A}$.
\end{teo}
\begin{proof}
For the equivalence of \ref{1} and \ref{2} see \cite[Equation 
(3.8)]{Azcarraga}).

Suppose that the bracket $\{\cdot,\cdot \}$ satisfy the graded Jacobi identity, that is,
\begin{equation*}
\{f,\{g,h\}\}=\{\{f,g\},h\}+(-1)^{\parity{f}\parity{g}}\{g,\{f,h\}\}
\end{equation*}
for all $f,g\in\mathcal{A}$. With this,
\begin{equation*}
\begin{split}
[\rho(\d{f}),\rho(\d{g})](h)&
=\rho(\d{f})(\rho(\d{g})(h))-
(-1)^{\parity{f}\parity{g}}\rho(\d{g})(\rho(\d{f})(h))\\
&=\{f,\{g,h\}\}-(-1)^{\parity{f}\parity{g}}\{g,\{f,h\}\}\\
&=\{\{f,g\},h\}\\
&=\rho(\d\{f,g\})(h).
\end{split}
\end{equation*}
Now suppose that \ref{3} is true, then
\begin{equation}\label{nada}
\begin{split}
\{\{f,g\},h\}&=\rho(\d\{f,g\})(h)=[\rho(\d{f}),\rho(\d{g})](h)\\
&=\{f,\{g,h\}\}-(-1)^{\parity{f}\parity{g}}\{g,\{f,h\}\}
\end{split}
\end{equation}
 \end{proof}

This result motivates the following:
\begin{defn}
	A Lie superalgebroid $(\Omega^1(\mathcal{A}), \lsem \cdot,\cdot \rsem, \rho)$ is of 
	\textbf{Poisson 
	type} if:
	\begin{enumerate}
	\item the anchor map $\rho$ is skew-supersymmetric, in the sense of 
	\eqref{eq:superskew}; and
	\item  both the bracket and anchor map satisfy any of the conditions of 
	Theorem \ref{teo:sn}.
	\end{enumerate}
\end{defn} 

In other words, a Lie superalgebroid $(\Omega^1(\mathcal{A}), \lsem \cdot,\cdot \rsem, 
\rho)$ is of Poisson type if it determines a Poisson superalgebra on $\mathcal{A}$.

Another important issue for us is to determine the form of the bracket of a Lie 
superalgebroid of
Poisson type. We know \cite[Example 2.4.5]{Mehta} that the classical example of Poisson 
supermanifolds 
leads to brackets of the type

\begin{equation*}
\pi^\sharp(\alpha) \ins \d\beta-
(-1)^{\parity{\alpha}\parity{\beta }}\pi^\sharp(\beta) \ins 
\d\alpha+
\d(\beta(\rho(\alpha)))
\end{equation*} 
where $\pi$ is a Poisson structure on a supermanifold $\mathcal M$ and 
$\pi^\sharp\colon{T\mathcal{M}}\to{T^{*}\mathcal{M}}$ is the musical morphism 
associated to it. The following result characterises the class of such 
algebroids.

\begin{prop}\label{pro:formacorchete}
Let $(\Omega^1(\mathcal{A}), \lsem \cdot,\cdot \rsem, \rho)$ be a Lie 
superalgebroid with skew-supersymmetric anchor $\rho$, and $\lsem 
\mathbf{d}f, 
\mathbf{d}g\rsem = \d\{f, g\}$;
then the bracket of  the Lie superalgebroid is of the form 
\begin{equation}\label{eq:formacorchete}
\lsem \alpha,\beta \rsem=
\rho(\alpha)\ins\d\beta-
(-1)^{\parity{\alpha}\parity{\beta 
}}\rho(\beta)\ins \d\alpha+
\d(\beta(\rho(\alpha))).
\end{equation}
\end{prop} 
\begin{proof}
Since every element of $\Omega^1(\mathcal{A}) $ is an  $\mathcal A$-linear 
combination of elements of the 
form $\d{f_i}$ for some $f_i\in\mathcal{A}$, it is enough to prove the 
statement for 
$\alpha=f_1\d{g_1}$ and $\beta = f_2 \d{g_2}$:
\[
\begin{split}
\lsem\alpha,\beta\rsem&=\lsem f_1\d{g_1},f_2\d{g_2}\rsem\\
&=\rho(f_1\d{g_1})(f_2)\d{g_2}-
(-1)^{\parity{\alpha}\parity{\beta}}f_2\rho(\d{g_2})(f_1)\d{g_1}\\
&\quad 
-(-1)^{\parity{\alpha}\parity{\beta}+\parity{g_2}\parity{f_1}}f_2f_1[\d{g_2},\d{g_1}]
\end{split}
\]
\[
\begin{split}
&=f_1\{g_1,f_2\}\d{g_1}-(-1)^{\parity{\alpha}\parity{\beta}}f_2\{g_2,f_1\}\d{g_1}\\
&\quad+(-1)^{\parity{f_1}\parity{f_2}+\parity{g_1}\parity{f_2}}f_2f_1\d\{g_1,g_2\}\\
&=f_1\{g_1,f_2\}\d{g_1}\\
&\quad+(-1)^{1+\parity{f_2}\parity{g_2}}f_1\{g_1,g_2\}\d{f_2}\\
&\quad-(-1)^{\parity{\alpha}\parity{\beta}}f_2\{g_2,f_1\}\d{g_1}
\end{split}
\]
\[
\begin{split}
&\quad+(-1)^{1+\parity{g_1}\parity{f_2}+
	\parity{f_1}\parity{f_2}+
	\parity{f_1}\parity{g_2}+
	\parity{f_1}\parity{g_1}}f_2\{g_1,g_2\}\d{f_1}\\
&\quad+(-1)^{\parity{f_1}\parity{f_2}+
	\parity{g_1}\parity{f_2}}f_2f_1\d\{g_1,g_2\}\\
&\quad+(-1)^{\parity{f_2}\parity{g_2}}f_1\{g_1,g_2\}\d{f_2}\\
&\quad+(-1)^{\parity{g_1}\parity{f_2}+\parity{f_1}\parity{f_2}+
	\parity{f_1}\parity{g_2}+\parity{f_1}\parity{g_1}}f_2\{g_1,g_2\}\d{f_1}\\
&=\rho(f_1\d{g_1})\ins \d(f_2\d{g_2}) 
-(-1)^{\parity{f_1\d{g_1}}\parity{f_2\d{g_2}}}\rho(f_2\d{g_2})\ins 
\d(f_1\d{g_1})\\
&\quad+\d(f_2\d{g_2}(\rho(f_1\d{g_1})))\\
&=\rho(\alpha)\ins\d\beta-
(-1)^{\parity{\alpha}\parity{\beta}}\rho(\beta)\ins 
\d\alpha+\d(\beta(\rho(\alpha)))\qedhere
\end{split}
\]
\end{proof}

An equivalent statement for the hypothesis of Proposition 
\ref{pro:formacorchete} is:
\begin{prop}
If $(\Omega^1(\mathcal{A}), \lsem \cdot,\cdot \rsem, \rho)$ is a Lie 
superalgebroid with  skew-supersymmetric anchor $\rho$,  then the
following assertions are equivalent:
\begin{enumerate}
\item\label{i} $\lsem \d{f}, \d{g}\rsem = \d(\d{g}(\rho(\d{f}))) = \d\{f, g\}$ 
for all $f, g \in 
\mathcal{A},$
\item\label{ii} $\d\alpha = 0 = \d\beta$ implies $\lsem \alpha, \beta\rsem = 0$ 
for $\alpha, 
\beta \in 
\Omega ^1(\mathcal{A}).$
\end{enumerate}
\end{prop}
\begin{proof}
	First, let us assume that \ref{i} holds. By Proposition \ref{pro:formacorchete} we know 
that 
	the bracket is of the form
	$$	\lsem \alpha, \beta \rsem = \rho(\alpha)\ins \d\beta 
	-(-1)^{\parity{\alpha}\parity{\beta}}\rho(\beta)\ins \d\alpha + 
	\d(\beta(\rho(\alpha))),$$
	and hence condition \ref{ii} holds,
	\begin{equation*}
	\begin{split}
	\d\lsem\alpha,\beta\rsem &= 
	\d( \rho(\alpha)\ins\d\beta - 
	(-1)^{\parity{\alpha}\parity{\beta}}\rho(\beta)\ins\d\alpha + 
	\d(\beta(\rho(\alpha))))\\
	&=\d^{2}(\beta(\rho(\alpha))) = 0,
	\end{split}
	\end{equation*}	
	whenever $\d\alpha = 0 = \d\beta$.
	
Now let us assume that condition \ref{ii} holds. Define 
\begin{equation*}
C(\alpha, \beta) = \lsem \alpha, \beta\rsem - 
\d(\beta(\rho(\alpha)))
\end{equation*} 
for $\alpha, \beta \in\Omega^1 (\mathcal{A})$. Notice that $C$ is skew-supersymmetric and
$$\d{C}(\alpha, \beta) = 
\d(\lsem\alpha, \beta\rsem) - 
\d^2 (\beta(\rho(\alpha))) = 0,$$
for any closed $\alpha$ and $\beta$.

Note that
\begin{equation*}
\begin{split}
C(f_1\d{g_1},\d{g_2})&=\lsem f_1\d{g_1},\d{g_2}\rsem-
\d(dg_2(\rho(f_1\d{g_1})))\\
&= -(-1)^{\parity{f_1}\parity{g_2}+
	\parity{g_1}\parity{g_2}}\rho(\d{g_2})(f_1)\d{g_1}\\
&\quad-(-1)^{\parity{g_1}\parity{g_2}}f_1[\d{g_2},\d{g_1}]\\
&\quad -\d(f_1\rho(\d{g_1})(g_2))\\
&=-(-1)^{\parity{f_1}\parity{g_2}+\parity{g_1}\parity{g_2}}
\rho(\d{g_2})(f_1)\d{g_1}\\
&\quad-(-1)^{\parity{g_1}\parity{g_2}}f_1[\d{g_2},\d{g_1}]\\
&\quad -(-1)^{\parity{f_1}\parity{g_1}+\parity{f_1}\parity{g_2}}
\rho(\d{g_1})(g_2)\d{f_1}\\
&\quad -f_1\d(\rho(\d{g_1})(g_2))\\
&= f_1C(\d{g_1},\d{g_2})-
(-1)^{\parity{g_1}\parity{g_2}+\parity{f_1}\parity{g_2}}
\rho(\d{g_2})\ins 
\d(f_1\d{g_1})
\end{split}
\end{equation*}

If $f_1=g_1$, then 
\begin{equation*}
\begin{split}
C(g_1\d{g_1},\d{g_2})&=g_1C(\d{g_1},\d{g_2})\\
&\quad -(-1)^{\parity{g_1}\parity{g_2}+\parity{g_1}\parity{g_2}}
\rho(\d{g_2})\ins \d(g_1\d{g_1})\\
&=g_1C(\d{g_1},\d{g_2})
\end{split}
\end{equation*}
Now, applying the operator $\d$ and taking into account that $g_1\d{g_1}$ is 
closed, we get
\begin{equation}\label{eq:ld1}
\begin{split}
0 &= \d{C}(g_1\d{g_1}, \d{g_2})\\
& = \d(g_1C(dg_1, dg_2))\\
& = \d{g_1} \wedge C(\d{g_1}, \d{g_2}) + g_1\d{C}(\d{g_1}, \d{g_2})\\
& = \d{g_1} \wedge C(\d{g_1}, \d{g_2}).
\end{split}
\end{equation}

Interchanging the roles of $g_1$ and $g_2$, we have
\begin{equation}\label{eq:ld2}
0  = \d{g_2} \wedge C(\d{g_2}, \d{g_1} ) = 
-(-1)^{\parity{g_1}\parity{g_2}}\d{g_2}	
\wedge C(\d{g_1}, \d{g_2}).
\end{equation}

Relations \eqref{eq:ld1} and \eqref{eq:ld2} imply that $\d{g_1}$, $\d{g_2}$ and 
$C(\d{g_1}, \d{g_2})$ 
are 
linearly dependent for arbitrary $g_1$ and $g_2$.
In particular, if $\d{g_1}$ and $\d g_2$ are linearly independent, then 
$C(\d{g_1}, \d{g_2}) = 0$, and 
hence
\[
\lsem \d{g_1},\d{g_2}\rsem = 
\d(\d{g_2}(\rho(\d{g_1} ))) = \d\{g_1, g_2\}
\]
which proves the Proposition.
\end{proof}

As an immediate consequence of all the work made above, we obtain:
\begin{teo}
	Let $\mathcal M$ be the cotangent supermanifold on the smooth manifold $M$, 
	and suppose $(\Omega(M),\{\cdot,\cdot\})$ is a Poisson superalgebra. Under 
	these conditions, the cotangent superbundle 
	$T^{*}\mathcal{M}$ inherits a Lie superalgebroid structure whose anchor map 
	is given by $\rho(\d{f})=\{f,\cdot\}$ and whose bracket is given by 
	\eqref{eq:formacorchete}.
\end{teo}

\section{Further structures on superforms}\label{sec:gerst}

Recall we defined on \eqref{eq:superd} the map 
$\d\colon\mathcal{A}\to\ker{\mathbf{m}}$ that associates to any $a$ a 
generator of the latter space, which we have denoted as 
$\Omega^{1}(\mathcal{A})$; now 	we set $\Omega^{0}(\mathcal A)=\mathcal A$ and 
define
\[
\Omega^{r}(\mathcal A):=\left(
\Omega^{1}(\mathcal A)
\right)^{r}
\quad\text{and}\quad
\Omega(\mathcal A):=\bigoplus_{k\geq 0}\Omega^{k}(\mathcal A);
\]
so defined, the space $\Omega(\mathcal A)$ is a bigraded-commutative algebra; 
if $\mathcal{A}=\Omega(M)$ 
then $\Omega(\mathcal A)$ is naturally isomorphic to the superalgebra of graded 
forms on the supermanifold $\supermap{M}{TM}$ ---see, for instance, 
\cite[Section 2]{monterde}---.

Let now $(\mathcal{A},\{\cdot, \cdot\})$ be a Poisson superalgebra. We define a 
bracket $[\cdot,\cdot]$ for elements of  $\Omega^{1}(\mathcal A)$ by the 
following relations:
\begin{equation}\label{eq:KS}
\begin{split}
[f,g]&=0\\
[f,\d{g}]&=[\d{f},g]=\{f,g\}\\
[\d{f},\d{g}]&=\d\{f,g\}
\end{split}
\end{equation}
for all $f, g, h\in \mathcal{A}$, and we extend to all of $\Omega 
(\mathcal{A})$  by 
\begin{equation*}
[\alpha,\beta\wedge\gamma]=
[\alpha,\beta]\wedge\gamma+
(-1)^{\parity{\alpha}\parity{\beta}}
\beta\wedge[\alpha,\gamma]
\end{equation*}
to obtain the following structure

\begin{defn}
	Let $(\mathcal{B},\wedge)$ be a $\Z$-graded superalgebra; if there is a 
	bracket $[\cdot,\cdot]\colon\mathcal{B}\otimes\mathcal{B}\to\mathcal{B}$ 
	such that
	\begin{enumerate}
	\item  
	$[\alpha,\beta]=-(-1)^{(|\alpha|-1)(|\beta|-1)+
		\parity{\alpha}\parity{\beta}}
	[\beta,\alpha]$
	\item 
	$[\alpha,[\beta,\gamma]]=[[\alpha,\beta],\gamma]+
	(-1)^{(|\alpha|-1)(|\beta|-1)+\parity{\alpha}\parity{\beta}}
	[\beta,[\alpha,	\gamma]]$
	\item 
	$[\alpha,\beta\wedge\gamma]=[\alpha,\beta]\wedge\gamma+
	(-1)^{(|\alpha|-1)|\beta|+\parity{\alpha}\parity{\beta}}
	\beta\wedge[\alpha,\gamma]$
\end{enumerate}
for all $\alpha, \beta, \gamma \in \mathcal{B}$, then $(\mathcal{B}, 
\wedge,[\cdot,\cdot])$ is a \textbf{Gerstenhaber superalgebra}.
\end{defn}

The following is a generalisation of a well-known result (see \cite{xu} and 
references therein):

\begin{teo}
	$\Omega(\mathcal{A})$ is a Gerstenhaber superalgebra if and only if 
	$\Omega^1(\mathcal{A})$ is a Lie superalgebroid.
\end{teo}
\begin{proof}
	Suppose that there is Lie superbracket $[\cdot, \cdot]$ of degree $-1$  
	that makes 
	$\Omega(\mathcal{A})$ into a Gerstenhaber superalgebra. It is clear that 
	$(\Omega^1(\mathcal{A}),\lsem\cdot, \cdot\rsem)$ is a Lie superalgebra with 
	$\lsem\alpha,\beta\rsem=[\alpha,\beta]$ for $\alpha$ and 
	$\beta\in\Omega^1(\mathcal{A})$.
	Now, for any $\alpha\in \Omega^1(\mathcal{A})$ and $f,g\in\mathcal{A}$, it 
	follows 
	from the properties of Gerstenhaber superalgebra that
	\begin{equation*}
	[\alpha, f g] = [\alpha, f ]g +(-1)^{\parity{\alpha}\parity{f}} f [\alpha, 
	g].
	\end{equation*}	
	Hence, $[\alpha,\cdot]$ is a derivation on $\mathcal{A}$. If we define 
	$\rho:\Omega^1(\mathcal{A})\rightarrow\Der(\mathcal{A})$ by 
	$\rho(\alpha)=[\alpha,\cdot]$, then 
	\begin{equation*}
	\lsem\alpha, f \beta\rsem = \rho(\alpha)(f) \beta + f \lsem\alpha, 
	\beta\rsem
	\end{equation*}	
	for $\alpha,\beta\in \Omega^1(\mathcal{A})$ and $f\in\mathcal{A}$.
	This shows that $(\Omega^1(\mathcal{A}), \lsem\cdot, \cdot\rsem,\rho)$ is 
	indeed a Lie 
	superalgebroid.\\
	Conversely, given a Lie superalgebroid $(\Omega^1(\mathcal{A}), \lsem\cdot, 
	\cdot\rsem,\rho)$,we define
	\begin{equation*}
	\begin{split}
	[f,g]&=0\\
	[\alpha,f]&=\rho(\alpha)(f)\\
	[\alpha,\beta]&=\lsem\alpha,\beta\rsem
	\end{split}
	\end{equation*}	
	for $f,g\in\mathcal{A}$ and $\alpha,\beta\in \Omega^1(\mathcal{A})$; for 
	decomposable forms $\mathbf{A}=\optupla{\alpha}{\wedge}{k}$ and 
	$\mathbf{B}=\optupla{\beta}{\wedge}{l}$ we define
	\begin{equation*}
	[\mathbf{A},\mathbf{B}]
	=\sum_{i=1,j=1}^{k,l}(-1)^{a_{i}+b_{j}+k-1} 
	\mathbf{A}_{i}
	\wedge[\alpha_i,\beta_j]\wedge
	\mathbf{B}_{j},
	\end{equation*}
	where we have defined: 
	$\mathbf{A}_{i}:=\omitupla{\alpha}{\wedge}{i}{k}$, 
	$\mathbf{B}_j:=\omitupla{\beta}{\wedge}{j}{l}$,  
	$a_{i}:=\parity{\alpha_i} \sum_{p=i+1}^{k}\parity{\alpha_p}+i$ and
	$b_{j}:=\parity{\beta_j}\sum_{q=1}^{j-1}\parity{\beta_q}+j$; the notation 
	$\omitupla{\alpha}{\wedge}{i}{k}$ means the $i$th factor is omitted. A long 
	computation shows that $\Omega(\mathcal{A})$ with $[\cdot,\cdot]$ is indeed 
	a Gerstenhaber superalgebra.
\end{proof}

\begin{defn}
	Let $[\cdot,\cdot]$ be a bracket on $\Omega(\mathcal{A})$ such that 
	$(\Omega(\mathcal{A}),\wedge,[\cdot,\cdot])$ is a Gerstenhaber 
	superalgebra. If the exterior superderivative is a derivation of 
	degree 1 for $[\cdot,\cdot]$, that is if
	\begin{equation*}
	\d[\alpha,\beta]= 
	[\d\alpha,\beta]+(-1)^{|\alpha|-1}[\alpha,\d\beta]
	\end{equation*}
	for all $\alpha, \beta \in 
	\Omega(\mathcal{A})$, we call it a \textbf{differential Gerstenhaber 
		superalgebra}.
\end{defn}

\begin{teo}
	A Gerstenhaber superbracket on $\Omega(\mathcal{A})$ is differential if and 
	only if it is the bracket associated to a Poisson superstructure on 
	$\mathcal{A}$ in the sense 
	of \eqref{eq:KS}.
\end{teo}
\begin{proof}
	Let $[\cdot,\cdot]$ be a differential Gerstenhaber superbracket on 
	$\Omega(\mathcal{A})$. We 
	define the map 
	$\{\cdot,\cdot\}:\mathcal{A}\times\mathcal{A}\to\mathcal{A}$ by  
	$(f,g)\mapsto \{f,g\}:=[f,\d 
	g]$ which is a Poisson superbracket; indeed:
	\begin{enumerate}
		\item $\{f,g\}=-(-1)^{\parity{f}\parity{g}}\{g,f\}$ because 
		\begin{equation*}
		\begin{split}
		0=&\d[f,g]=[\d{f},g]+(-1)^{(|f|-1)}[f,\d{g}]\\
		&=-(-1)^{(|\d{f}|-1)(|g|-1)+\parity{\d{f}}\parity{g}}[g,\d{f}]-
		[f,\d{g}]\\
		&=-(-1)^{\parity{f}\parity{g}}\{g,f\}-\{f,g\};
		\end{split}
		\end{equation*}
		\item the graded Jacobi identity holds:
		\begin{equation*}
		\begin{split}
		\{f,\{g,h\}\}&=[f,\d[g,\d{h}]]\\
		&= [f,[\d{g},\d{h}]+(-1)^{(|g|-1)}[g,\d^2h]]\\
		&=[[f,\d{g}],\d{h}]+
		(-1)^{(|f|-1)(|\d{g}|-1)+\parity{f}\parity{\d{g}}}
		[\d{g},[f,\d{h}]]\\
		&=\{\{f,g\},h\}+(-1)^{\parity{f}\parity{g}}\{g,\{f,h\}\}\\
		\end{split}
		\end{equation*}
		\item the graded Leibniz identity holds:
		\begin{equation*}
		\begin{split}
		\{f,g\wedge h\}&=[f,\d(g\wedge h)]=[f,\d g\wedge h+g\wedge dh]\\
		&=[f,\d g\wedge h]+[f, g\wedge \d h]\\
		&=[f,\d g]\wedge h+(-1)^{(|f|-1)|\d g|+\parity{f}\parity{\d 
		g}}\d{g}\wedge[f,h]\\
		&\quad+[f,g]\wedge 
		\d{h}+(-1)^{(|f|-1)|g|+\parity{f}\parity{g}}g\wedge[f,\d h]\\
		&=\{f,g\}\wedge h+(-1)^{\parity{f}\parity{g}}g\wedge\{f,h\}		
		\end{split}
		\end{equation*}
	\end{enumerate}
	So, $\{\cdot,\cdot\}$ is a Poisson superbracket on $\Omega(M)$; moreover 
	$\{\cdot,\cdot\}$  
	induces $[\cdot,\cdot]$ by the relations \eqref{eq:KS}.
	
	Now, let $[\cdot,\cdot]:\Omega(\mathcal{A})\times \Omega(\mathcal{A})\rightarrow 
	\Omega(\mathcal{A})$ be the bracket associated to a Poisson superalgebra
	$(\mathcal{A},\{\cdot,\cdot\})$. It is enough show that it is differential for 
	$\alpha=\d f$ 
	and 
	$\beta=g \d h$ for some $f,g,h\in\mathcal{A}$:
	\begin{equation*}
	\begin{split}
	\d[\alpha,\beta]&=\d[\d f,g\d h]\\
	&=\d([\d f,g] 
	h+(-1)^{(\abs{\d f}-1)(\abs{g}-1)+\parity{\d f}\parity{g}}g 
	[\d f,\d h])\\
	&=\d[\d f,g]\wedge \d h+
	(-1)^{\parity{f}\parity{g}}\d g\wedge[\d f,\d h]\\
	&=\d\{f,g\}\wedge \d 
	h+(-1)^{(\abs{\d{f}}-1)(|g|-1)+\parity{df}\parity{\d g}}
	\d{g}\wedge[\d f,\d h]\\
	&=[\d f,\d g]\wedge\d h
	+(-1)^{(|\d f|-1)(|g|-1)+\parity{\d f}\parity{\d g}}
	\d g\wedge[\d f,\d h]\\
	&=[\d f,\d g\wedge \d h]\\
	&=[\d^2f,g\d h]+(-1)^{|\d f|-1}[\d f,\d(g\d h)]\\
	&=[\d\alpha,\beta]+
	(-1)^{|\alpha|-1}[\alpha,\d\beta]\qedhere
	\end{split}
	\end{equation*}
\end{proof}

Let $(\Omega(\mathcal{A}),\wedge,[\cdot,\cdot],\d)$ be a differential 
Gerstenhaber 
superalgebra; from the Theorem above,
$(\Omega^1(\mathcal{A}),\lsem\cdot,\cdot\rsem,\rho)$ is a Lie superalgebroid.
Moreover, $\rho$ is skew-supersymmetric since
\begin{equation*}
\begin{split}
f_1\d{g_1}(\rho(f_2\d{g_2}))
&=\langle
f_2\rho(\d{g_2}),f_1\d{g_1}
\rangle\\
&=(-1)^{\parity{f_1}\parity{g_1}}f_2[\d g_2,\d g_1]f_1\\
&=-(-1)^{\parity{f_1}\parity{g_1}+\parity{g_1}\parity{g_2}}f_2[\d g_1,\d g_2]f_1\\
&=-(-1)^{\parity{f_1}\parity{g_1}+\parity{g_1}\parity{g_2}+\parity{f_2}\parity{[\d g_1,\d 
g_2]}+\parity{f_1}\parity{f_2}+\parity{f_1\parity{[\d g_1,\d g_2]}}}f_1[\d g_1,\d g_2]f_2\\
&=-(-1)^{\parity{f_1\d g_1}\parity{f_2\d g_2}}
\langle
f_1\rho(\d{g_1}),f_2\d{g_2}
\rangle\\
&=f_2\d{g_2}(\rho(f_1\d{g_1}))
%
\end{split}
\end{equation*}
for all $f_1,f_2,g_2,g_2\in\mathcal{A}$, where $\langle\cdot, \cdot\rangle$ denotes the pairing of 
supervectorfields and  $1$-superforms. Also, 
due to the fact that the bracket is differential
we have 
\begin{equation*}
\d(\rho(\d f)(g))=\d([\d f,g])=[\d f,\d g].
\end{equation*}
That is, the hypotheses of Proposition \ref{pro:formacorchete} are satisfied, therefore 
the bracket of the Lie superalgebroid is of the form \eqref{def:corchete} and 
therefore the 
superalgebroid is of Poisson type.  Now, we can prove  the following

\begin{prop}
	If $(\Omega^1(\mathcal{A}),\lsem\cdot,\cdot\rsem,\rho)$ is a Lie algebroid 
	with 
	\begin{enumerate}
		\item $\rho$ skew-supersymmetric
		\item  $\lsem\d f,\d g\rsem=\d(\rho(\d f)(g))$ for all $f,g\in\mathcal{A}$
	\end{enumerate}
	then $(\Omega(\mathcal{A}),\wedge,[\cdot,\cdot],\d)$ is a differential 
	Gerstenhaber superalgebra.
\end{prop}

\begin{proof}
We only need to prove that the Gerstenhaber superalgebra induced by the superalgebroid is 
differential. For this, first note that
\begin{equation*}
\lsem\d f, \d g\rsem=\d(\rho(\d f)(g))=\d [\d f,g]
\end{equation*} 
for all $f,g\in\mathcal{A}$ and remember that it is enough show that the bracket 
induced by $\lsem\cdot,\cdot\rsem$ in $\Omega(\mathcal{A})$ it is differential for 
$\alpha=\d f$ and $\beta=g \d h$ for some $f,g,h\in\mathcal{A}$:
\begin{equation*}
\begin{split}
\d[\d f, g\d h]&=\d\lsem\d f, g\d h\rsem\\
&=\d (\rho(\d f)(g)\d h+(-1)^{\parity{\d f}\parity{g}}g\lsem\d f,\d h\rsem)\\
&=\d (\rho(\d f)(g))\wedge\d h+(-1)^{\parity{\d f}\parity{g}}\d g\wedge\lsem\d f,\d 
h\rsem+(-1)^{\parity{\d f}\parity{g}} g\d \lsem\d f,\d h\rsem\\
&=\lsem\d f,\d g\rsem\wedge\d h+(-1)^{\parity{ f}\parity{g}} \d g\wedge \lsem\d f,\d 
h\rsem+(-1)^{\parity{\d f}\parity{g}} g \d^2[\d f,\d h]\\
&=[\d f,\d g]\wedge\d h+(-1)^{\parity{\d f}\parity{g}} \d g\wedge [\d f,\d 
h]\\
&=[\d f,\d(g\d h)].
\end{split}
\end{equation*}
This establishes the proposition.
\end{proof}

Observe that if $ (\mathcal{A}, \{\cdot, \cdot\}) $ is a Poisson superalgebra, 
we know from the 
previous 
section that $ (\Omega^1 (\mathcal{A}), \lsem \cdot, \cdot \rsem, \rho) $ is a 
Lie superalgebroid of Poisson type with bracket defined by $ 
\eqref{def:corchete} $. All of the above implies that $ \lsem 
\cdot, \cdot 
\rsem $ coincides with the restriction on $1$-superforms of the bracket defined 
by $ 
\eqref{eq:KS}$. We have now proved:

\begin{teo}
$(\mathcal{A}, \{\cdot,\cdot\})$ is a Poisson superalgebra if and only if 
$(\Omega^1(\mathcal{A}),\lsem\cdot,\cdot\rsem,\rho)$	is a Lie superalgebroid of Poisson 
type with bracket defined by $\eqref{eq:formacorchete}$.
\end{teo}

\noindent\textbf{Acknowledgements:} the authors wish to thank the Mexican 
National Council for Science and Technology (CONACYT) for the funding of 
their Postdoctoral Fellowships under programme 2018-000005-01-NACV 
---\emph{Programa de Estancias Posdoctorales Nacionales 2018-1}--- with 
reference numbers (CVU) 289063 and 344289, during which this paper was written.

\bibliographystyle{plain}
\bibliography{superpoiss}

\begin{description}
	\item[Email addresses:] dennise.gb@fc.uaslp.mx, oscar.guajardo@uaslp.mx
	\item[Postal address:] Facultad de Ciencias, UASLP. Av. Chapultepec 1570, Privadas del 
	Pedregal. C.P. 78295, San Luis Potosí, SLP, México. 
\end{description}

\end{document}